\documentclass[10pt,a4paper]{amsart}
\usepackage{amsfonts,amsmath,amssymb}

\newtheorem{dummy}{realdumb}

\newtheorem{theorem}[dummy]{Theorem}

\theoremstyle{definition}               

\newtheorem*{definition}{Definition}
\newtheorem*{remark}{Remark}
\newtheorem*{remarks}{Remarks}

\DeclareMathOperator{\ind}{Ind}
\newcommand{\bbZ}{\mathbb Z}
\newcommand{\bbR}{\mathbb R}

\begin{document}

\title{Local Indices of a Vector Field at an Isolated Zero on the Boundary}

\author[H. Kamae]{Hiroaki Kamae}
\address{Department of Applied Science, Graduate School of Science, Okayama University of Science, Okayama, Okayama 700-0005, Japan (from April 2008)}

\author[M. Yamasaki]{Masayuki Yamasaki}
\address{Department of Applied Science, Faculty of Science, Okayama University of Science, Okayama, Okayama 700-0005, Japan}
\email{yamasaki@surgery.matrix.jp}

\begin{abstract}
We define two types of local indices of a vector field at an isolated
zero on the boundary, and 
prove  Poincar\'e-Hopf-type index theorems for certain
vector fields on a compact smooth manifold which have only isolated zeros.
\end{abstract}

\thanks{The second author was partially supported by Grant-in-Aid for
Scientific Research.}

\subjclass[2000]{Primary 57R25; Secondary 55M25}
\keywords{vector field, index, Euler characteristic}

\maketitle

\section{Introduction}
The famous Poincar\'e-Hopf theorem states that the index $\ind(V)$ 
of a continuous tangent vector field $V$ on a compact smooth manifold $X$
is equal to the Euler charactersitic $\chi(X)$ of $X$, if $V$ has only isolated zeros
away from the boundary and $V$ points outward on the boundary of $X$.
If you assume that the vectors on some of the boundary components point
inward and point outward on the other components, then the formula will look like:
\[
\ind(V) = \chi(X) - \chi(\partial_-X)~,
\]
where $\partial_-X$ denotes the union of the boundary components on which
the vectors point inward.  This can be observed by looking at the Morse function
of the pair $(X,\partial_-X)$.
In \cite{Morse}, M. Morse 
relaxed the requirement on the boundary behavior and obtained a formula
\[
\ind(V)+\ind(\partial_-V)=\chi(X)~.
\]
Actually the requirement that the singularities are isolated are also
relaxed.  
This formula has been rediscovered and extended by several authors
\cite{Pugh} \cite{BG} \cite{GS}.
In this paper we consider only vector fields whose zeros are isolated.
But we allow zeros on the boundary.

Let $X$ be an $n$-dimensional compact smooth manifold with boundary $\partial X$, 
and fix a Riemanian metric on $X$.  We assume $n\ge 1$.
For a continuous tangent vector field $V$ on $X$ and a point $p$ of its
boundary, we define the vector $\partial V(p)$ to be the orthogonal projection
of $V(p)$ to the tangent space of $\partial X$ at $p$. The tangent vector field
$\partial V$ on $\partial X$ is called the {\sl boundary} of $V$.
$\partial^\perp V$ denotes the normal vector field on $\partial X$ defined by
$\partial^\perp V(p)=V(p)-\partial V(p)$.
A zero $p$ of $\partial V$ is said to be of {\sl type $+$}
if $V(p)$ is an outward vector.
It is of {\sl type $-$} if $V(p)$ is an inward vector.
It is of {\sl type $0$} if it is also a zero of $V$.

Suppose $p$ is an isolated zero of $V$.  If $p$ is in the interior of $X$, 
then the {\it local index} $\ind(V,p)$ of $V$ at $p$ is defined 
as is well known; it is an integer.
When $p$ is on the boundary and is an isolated zero of $\partial V$, 
we will define the {\it normal local index} $\ind_\nu(V,p)$ of $V$ at $p$
which is either an integer or a half-integer in the next section;
when $p$ is  an isolated zero of $\partial^\perp V$, 
we will define the {\it tangential local index} $\ind_\tau(V,p)$ of $V$ at $p$.
This may be a half-integer, too, when $n\le2$.
These two local indices are not necessarily the same
when they are both defined.

When the zeros of $V$ and $\partial V$ are all isolated, 
we  define the {\it normal index} $\ind_\nu (V)$ of $V$ to be the sum
of the local indices at the zeros in the interior and the normal
local indices at the zeros on the boundary.
The sum of the local indices of $\partial V$ at the zeros of type $+$ ({\it resp.}
$-$, $0$) is denoted $\ind(\partial_+V)$ ({\it resp.} $\ind(\partial_-V)$,
$\ind(\partial_0V)$).
\begin{theorem}
Suppose $X$ is an $n$-dimensional compact smooth manifold
and $V$ is a continuous tangent vector field on $X$.
If $V$ and $\partial V$ have only isolated zeros, then the following equality holds:
\[
\ind_\nu(V)+\frac{1}{2}\ind(\partial_0 V) + \ind(\partial_- V) =\chi(X)~.
\]
\end{theorem}
\begin{remarks} (1) The local index of a zero of the zero vector field
on a 0-dimensional manifold is always 1.  So, when $n=1$,
$\ind(\partial_0 V)$ is the number of the zeros on the boundary, and
$\ind(\partial_- V)$ is the number of boundary points at which the vector
points inward.\\
(2) The special case where the vectors $V(p)$ are tangent to the boundary
for all $p\in \partial X$ were discussed in \cite{MW};
see the review by J. M. Boardman in Mathematical Reviews.
\end{remarks}

When the zeros of $V$ are isolated and the zeros of $V$ on the boundary
are the only zeros of $\partial^\perp V(p)$, 
we will define the {\it tangential index} $\ind_\tau (V)$ of $V$ to be the sum
of the local indices of $V$ at the zeros in the interior and the tangential
local indices at the zeros on the boundary.
If the dimension of $X$ is bigger than 2, then the assumption on $V$
forces the connected components of the boundary of $X$ to be classified into the
following two types:
\begin{enumerate}
\item vectors point outward except at the isolated zeros,
\item vectors point inward except at the isolated zeros.
\end{enumerate}
The union of the components of the first type is denoted $\partial_+X$,
and the union of the components of the second type is denoted $\partial_-X$.
If the dimension of $X$ is 1, then the boundary components are single points;
so the vector at the boundary either points outward, inward, or is zero,
and accordingly the boundary $\partial X$ is split into $\partial_+X$,
$\partial_- X$, and $\partial_0X$.

\begin{theorem}
Suppose $X$ is an $n$-dimensional compact smooth manifold
and $V$ is a continuous tangent vector field on $X$.
If the zeros of $V$ are isolated and the zeros of $V$ on the boundary
are the only zeros of $\partial^\perp V(p)$, 
then the following equality holds:
\[
\ind_\tau(V)=
\begin{cases}
\chi(X) & \text{if $n$ is even},\\
\chi(X) - \chi(\partial_- X) & \text{if $n\ge 3$},\\
\chi(X)-\frac{1}{2}\chi(\partial_0 X) - \chi(\partial_- X) & \text{if $n=1$}.
\end{cases}
\]
\end{theorem}

In the last section, we will give an alternative formulation of these theorems.

\section{Local Indices of an Isolated Zero on the Boundary}
In this section, we describe the two local indices of a vector field $V$
at an isolated zero on the boundary.  

Let $X$ be an $n$-dimensional compact smooth manifold with boundary $\partial X$.
We fix an embedding of $\partial X$ in a Euclidean space $\bbR^N$ of a sufficiently
high dimension so that, under the identification
$\bbR^N=1\times \bbR^N$, it extends to an an embedding of $(X,\partial X)$ in 
$([1,\infty)\times \bbR^{N}, 1\times \bbR^{N})$ such that
$X\cap [1,2]\times \bbR^N=  [1,2]\times\partial X $.

Now suppose $p$ is an isolated zero sitting on the boundary $\partial X$.
Let us take local cordinates $y_1$, $y_2$, \dots, $y_n$ around $p$
such that $y_1$ is equal to the first coordinate of $[1,\infty)\times \bbR^N$
and $p$ corresponds to $a=(1,0,\dots,0)\in \bbR^n$.
$V$ defines a vector field $v$ on a neighborhood of $a$ in the subset $y_1\ge 1$.
Choose a sufficiently small positive number $\varepsilon$ so that
the right half $D^n_+(a;\varepsilon)$ of the disk of radius $\varepsilon$
with center at $a$ is contained in this neighborhood,
and $a$ is the only zero of $v$ in $D^n_+(a;\varepsilon)$.
Let $H^{n-1}_+(a;\varepsilon)$ ($\subset\partial D^n_+(a;\varepsilon)$)
denote the right hemisphere of radius
$\varepsilon$ with center at $a$.  The vector field $v$ induces a continuous map
$\bar v:H^{n-1}_+(a;\varepsilon)\to S^{n-1}$ to the $(n-1)$-dimensional 
unit sphere by:
\[
\bar v(x)=\frac{v(x)}{\Vert v(x)\Vert}~.
\]
Let $S^{n-2}(a;\varepsilon)$ denote the boundary sphere 
of $H^{n-1}_+(a;\varepsilon)$.
When $n=1$, we understand that it is an empty set.
Assume that its image by $\bar v$ is not the whole sphere $S^{n-1}$.
Pick up a ``direction'' $d\in S^{n-1}\setminus \bar v (S^{n-2}(a;\varepsilon))$,
then $\bar v$ determines an integer, denoted $i(v,a;d)$, in 
$H_{n-1}(S^{n-1}, S^{n-1}\setminus\{d\})=\bbZ$.  
Here we use the compatible orientations for $H^{n-1}_+(a;\varepsilon)$
and $S^{n-1}$. 
It is the algebraic intersection number of $\bar v$ with $\{d\}\subset S^{n-1}$,
 and is locally constant as a function of $d$.
A pair of antipodal points $\{d,-d\}$ of $S^{n-1}$ is said to be
{\it admissible} if they are both in 
$S^{n-1}\setminus \bar v (S^{n-2}(a;\varepsilon))$.
For such an admissible pair $\{\pm d\}$, 
we define a possibly-half-integer $i(v,a;\pm d)$ to be the average
of the two integers $i(v,a;d)$ and $i(v,a;-d)$:
\[
i(v,a;\pm d)=\frac{1}{2} i(v,a;d) + \frac{1}{2}i(v,a;-d)~.
\]
In the case of $n=1$, there is only one admissible pair $\{\pm 1\}=S^0$,
and 
\[
i(v,1;\pm 1)=
\begin{cases}
\frac{1}{2} & \text{if $\bar v(1+\varepsilon)=1$,}\\
-\frac{1}{2} & \text{if $\bar v(1+\varepsilon)=-1$} .
\end{cases}
\]
\begin{definition}
Suppose $p$ is an isolated zero of $\partial V$. 
We may assume that $\varepsilon$ is sufficiently small,
and that the pair $\{\pm e_1\}$ with $e_1=(1,0,\dots,0)\in S^{n-1}$ is admissible.
The {\it normal local index} $\ind_\nu(V,p)$ of $V$ at $p$
is defined to be $i(v,a;\pm e_1)$.
\end{definition}

\begin{definition}
Suppose $p$ is an isolated zero of $\partial^\perp V$.
We define the {\it tangential local index} $\ind_\tau(V,p)$ of $V$ at $p$
as follows:  If $n=1$, then $\ind_\tau(V,p)=i(v,1;\pm 1)$.
If $n\ge 2$, then set $S^{n-2}=\{e\in S^{n-1}|e\perp (1,0,\dots,0)\}$.
We may assume that $\varepsilon$ is sufficiently small,
and that, $S^{n-2}\subset S^{n-1}\setminus \bar v (S^{n-2}(a;\varepsilon))$.
When $n=2$, there is only one admissible pair in $S^{n-2}=S^0$.
When $n\ge 3$, the value of $i(v,a;d)$ is independent of the choice
of $d\in S^{n-2}$, and $i(v,a;\pm d)=i(v,a;d)$.
So, for $n\ge 2$, we define $\ind_\tau(V,p)$ to be $i(v,a;\pm d)$,
where $d$ is any point in $S^{n-2}$.
\end{definition}

\begin{remarks} (1) When $n=1$, the two indices are the same.\\
(2) When $n\ge 3$, $\ind_\tau(V,p)$ is an integer.
\end{remarks}

\section{Proof of Theorem 1}
We give a proof of Theorem 1.  
Assume that $(X,\partial X)$ is embedded in
$([1,\infty)\times \bbR^{N}, 1\times \bbR^{N})$ as in the previous section.
We consider the double $DX$ of $X$:
\[
DX=\partial ([-1,1]\times X)=\{\pm 1\}\times X \cup [-1,1]\times \partial X~.
\]
$DX$ can be embedded in $\bbR\times\bbR^N$ as the union of
three subsets $X_+$, $X_-$, $[-1,1]\times\partial X$, where
$X_+$ is $X$ itself, $X_-$ is the image of the reflection $r:\bbR\times
\bbR^N\to \bbR\times \bbR^N$ with respect to $0\times\bbR^N$, 
and $\partial X\subset 1\times \bbR^N$
is regarded as a subset of $\bbR^N$.

Let $V=V_+$ be the given tangent vector field on $X=X_+$.
The reflection $r$ induces a tangent vector field $r_*(V)=V_-$ on $X_-$.
We can extend these to obtain  a tangent vector field $DV$ on $DX$ by 
defining $DV(t,x)$ to be
\[
\frac{t+1}{2}V_+(1,x) +\frac{1-t}{2}V_-(-1,x)
\]
for $(t,x)\in [-1,1]\times \partial X$.
Note that, on $0\times \partial X$, we obtain the boundary $\partial V$ of $V$.
There are four kinds of zeros of $DV$:
\begin{enumerate}
\item For each zero $p$ of $V$ in the interior of $X$, there are two zeros: 
the copy in the interior of $X_+$ and the copy in the interior of $X_-$.
They have the same local index as the original one.
\item For each zero $p=(1,x)$ of $\partial V$ of type 0, 
the points $(t,x)$ are all zeros of $DV$.  Although these are not isolated,
we can perturb the vector field in a very small neighborhood and make it
into an isolated zero, whose local index is $2\ind_\nu(V,p)$.
\item For each zero $p=(1,x)\in\partial X$ of $\partial V$ of type $-$,
the point $(0,x)$ is an isolated zero of $DV$ whose local index is 
equal to $\ind(\partial V,p)$.
\item For each zero $p=(1,x)\in\partial X$ of $\partial V$ of type $+$,
the point $(0,x)$ is an isolated zero of $DV$ whose local index is 
equal to $-\ind(\partial V,p)$.
\end{enumerate}
One can verify the computation of the local indices in cases (2), (3), and (4)
above as follows: First define the local coordinates $y_1$, \dots, $y_n$ 
around $(0,x)$ extending the $y_i$'s around $p=(1,x)$ described in \S2 by 
\[
\begin{cases}
y_1(t,*)=t~ & \text{for all $t\leq 1$} \\ 
y_i(t,x')=y_i(1,x') & \text{if  $i=2,\dots, n$ and $-1\leq t\leq 1$},\\
y_i(t, x'')=y_i(-t, x'') & \text{ if $i=2,\dots, n$ and $t\leq -1$ }.
\end{cases}
\]
Then consider the map 
\[
D\bar v: r(H_+^{n-1}(a;\varepsilon))\cup [-1,1]\times S^{n-2}(a;\varepsilon)\cup
H^{n-1}_+(a;\varepsilon)\to S^{n-1}
\]
induced from $DV$,
and compute the algebraic intersection number with $e_1=(1,0,\dots,0)$
in case (2) and with $e_2=(0,1,0,\dots, 0)$ in cases (3) and (4).
Note that (3) and (4) do not occur when $n=1$. 
Let $\bar v:H^{n-1}_+(a;\varepsilon)\to S^{n-1}$ be the map
induced by $V$ as in \S2. 
Note that $\bar v$ can be defined not only for an isolated zero of 
$\partial V$ of type 0 but also for a zero of type $\pm 1$.  
$D\bar v$ is the double of $\bar v$ in the sense that it is 
$\bar v$ on the subset $H^{n-1}_+(a;\varepsilon)$ and that it is 
the composite $r\circ \bar v \circ r$ on the subset
$r(H^{n-1}_+(a;\varepsilon))$; therefore, for $q\in r(H^{n-1}_+(a;\varepsilon))$, 
$D \bar v(q)=e_1$ if and only if $\bar v(r(q))=-e_1$.
In case (2), the vectors on the subset $[-1,1]\times S^{n-2}(a;\varepsilon)$
and $e_1$ are never parallel; 
so the algebraic intersection of $D\bar v$ with $e_1$ is
$i(v,a;e_1)+i(v,a;-e_1)=2\ind_\nu(V,p)$.  In case (3) ({\it resp.} (4)), 
we may assume that all the vectors $D\bar v((t,x'))$ ($t\ne 0$) point away from
({\it resp.} toward)
the hyperplane $y_1=0$; therefore, the local index is equal to
$\ind(\partial V, p)$ ({\it resp.} $-\ind(\partial V, p)$), 
since the $y_1$ direction is preserved ({\it resp.} reversed) in case (3)
({\it resp.} (4)).

Apply the Poincar\'e-Hopf index theorem to $DV$ and $\partial V$;
we obtain the following equalities:
\begin{align*}
2 \ind_\nu(V) + \ind(\partial_-V) -\ind(\partial_+V) &= 2\chi(X) -\chi(\partial X)~,\\
\ind(\partial_0 V)+ \ind(\partial_- V)+ \ind(\partial_+ V)&=\chi(\partial V)~.
\end{align*}
The desired formula follows immediately from these.

\section{Proof of Theorem 2}
When $n=1$, the normal local index and the tangential local index
are the same; therefore, the $n=1$ case follows from Theorem 1.
So we assume that $n\ge 2$.

Let $DX$ be the double of $X$ and let us use the same notation as
in the first paragraph of the previous section.
We will define the twisted double $\tilde DV$ of the vector field $V$ on $X$
as follows:
$\tilde V_+=V$ is a vector field on $X=X_+$.  
Consider $-V$; the reflection $r$ induces
a vector field $\tilde V_-=v_*(-V)$ on $X_-$.  
Extend these to obtain a tangent vector field $\tilde DV$ on $DX$ by 
defining $\tilde DV(t,x)$ to be
\[
\frac{t+1}{2}\tilde V_+(1,x) +\frac{1-t}{2}\tilde V_-(-1,x)
\]
for $(t,x)\in [-1,1]\times \partial X$.
In general, if $V(p)$ is tangent to $\partial X$ at $p=(1,x)\in \partial X$,
then the twisted double $\tilde DV$ has a corresponding zero $(0,x)$.
We are assuming that this happens only when $p$ is a zero of $V$.
Thus there are only two types of zeros of $\tilde DV$:
\begin{enumerate}
\item For each zero $p$ of $V$ in the interior of $X$, there are two zeros: 
the copy in the interior of $X_+$ which has the same local index as $\ind(V,p)$
and the copy in the interior of $X_-$ whose local index is equal to
$(-1)^n\ind(V,p)$.
\item For each zero $p=(1,x)$ of $V$ on the boundary of $X$, 
the points $(t,x)$ are all zeros of $\tilde DV$.  
Although these are not isolated, 
we can perturb the vector field in a very small neighborhood and make it
into an isolated zero, whose local index is equal to $2\ind_\tau(V,p)$ 
if $n$ is even  and is equal to 0 if $n$ is odd.
\end{enumerate}
The computation of the local index in case (2) can be done in the following way.
Let us use the notation in the previous section.  In this case
we consider
\[
\tilde D\bar v: r(H_+^{n-1}(a;\varepsilon))\cup [-1,1]\times S^{n-2}(a;\varepsilon)\cup
H^{n-1}_+(a;\varepsilon)\to S^{n-1}
\]
induced from $\tilde DV$,
and compute the algebraic intersection number with $e_2=(0,1,0,\dots, 0)$.
$\tilde D\bar v$ is the twisted double of $\bar v$ in the sense that it is 
$\bar v$ on the subset $H^{n-1}_+(a;\varepsilon)$ and that it is 
the composite $r\circ A\circ \bar v \circ r$ on the subset
$r(H^{n-1}_+(a;\varepsilon))$, where $A:S^{n-1}\to S^{n-1}$
is the antipodal map; therefore, for $q\in r(H^{n-1}_+(a;\varepsilon))$, 
$\tilde D \bar v(q)=e_2$ if and only if $\bar v(r(q))=-e_2$.
The vectors on the subset $[-1,1]\times S^{n-2}(a;\varepsilon)$
and $e_2$ are never parallel; 
so the algebraic intersection of $\tilde D\bar v$ with $e_2$ is
$i(v,a;e_1)+(-1)^n i(v,a;-e_1)$ which is equal to
$2\ind_\tau(V,p)$ if $n$ is even and is equal to 0 if $n$ is odd.

So, if $n$ is even, the Poincar\'e-Hopf formula for $\tilde DV$ reduces to
the desired formula $\ind_\tau V=\chi(X)$.

Next we consider the case where $n\ge 3$.
As we mentioned in the first section, the components of $\partial X$
are classified into two types:
\begin{enumerate}
\item vectors point outward except at the isolated zeros,
\item vectors point inward except at the isolated zeros.
\end{enumerate}
Suppose that $p$ is an isolated zero of $V$ on a connected component $C$
of $\partial X$ and that $C$ is of the first type. Consider a small neighborhood
of $p$ and coordinates $\{y_1, \dots, y_n\}$ as in \S2.
The vector field $v$ along $y_1=1$ can be thought of as a map 
$\varphi(y_2,\dots,y_n)=(z_1,z_2,\dots,z_n)$
from an open set $U\subset \bbR^{n-1}$ to $\bbR^n$
satisfying $z_1\ge 0$.  
The equality holds if and only if $(y_2,\dots,y_n)=(0,\dots,0)$.
Choose a very small number $\varepsilon>0$. Using a homotopy
\[
\max\{\varepsilon-(y_2^2+y_3^2+\dots+y_n^2),0\}(-t,0,\dots,0)+
\varphi(y_2,\dots,y_n)~,
\]
one can add a collar along $C$ and extend the vector field $V$
over the added collar.  Repeat this process
if there are more zeros on $C$ until the vector points outward along the
new boundary component.  The zeros on the boundary component $C$ now
lies in the interior, and the local indices are eaual to the corresponding
tangential local indices.
We can do a similar modification in the case of the second type
component, and move all the zeros on the boundary into the interior.
Now apply the Poincar\'e-Hopf theorem to get:
\[
\ind_\tau V =\chi(X)-\chi(\partial_-X)~.
\]
This completes the proof.

\section{An Alternative Formulation}
Let $V$ be a continuous vector field on an $n$-dimensional compact smooth 
manifold $X$ whose zeros are isolated.
In the previous sections, we considered the zeros of $V$ as the only singular
points, and defined the normal/tangential index as the sum of local indices
only at the zeros.  In this section, the zeros of $\partial V$ (in the normal index case)
and the zeros of $\partial^\perp V$ (in the tangential index case)
are also regarded as singular points of $V$.  
Note that the definition of the normal ({\it resp.} tangential) local index at an
isolated zero on the boundary given in \S2 is valid for 
an isolated zero of $\partial V$ ({\it resp.} $\partial^\perp V$).

\begin{definition}
When the zeros of $V$ and $\partial V$ are all isolated, 
the {\it expanded normal index} $\ind^*_\nu(V)$ of $V$ is defined to be
the sum of the local indices of $V$ at the interior zeros of $V$ and 
the normal local indices of $V$ at the zeros of $\partial V$.
When the zeros of $V$ and $\partial^\perp V$ are all isolated, 
the {\it expanded tangential index} $\ind^*_\tau(V)$ of $V$ is defined to be
the sum of the local indices of $V$ at the interior zeros of $V$ and 
the tangential local indices of $V$ at the zeros of $\partial^\perp V$.
\end{definition}

\begin{remark}
Note that the tangential local index  $\ind_\tau(V,p)$ at an isolated
zero $p$ of $\partial^\perp V$ is equal to zero if $n\ge 3$ and 
$p$ is not a zero of $V$; this can be observed by choosing $d\in S^{n-2}$
to be not equal to $\pm \bar v(p)$.
Also note that, if $n=1$, the zeros of $\partial^\perp V$ are 
automatically the zeros of $V$.
Therefore, $\ind^*_\tau(V)=\ind_\tau(V)$ if $n\neq 2$.
\end{remark}

\begin{theorem} 
Suppose $X$ is an $n$-dimensional compact smooth manifold
and $V$ is a continuous tangent vector field on $X$.
If $V$ and $\partial V$ have only isolated zeros, then the following equality holds:
\[
\ind^*_\nu(V)=\begin{cases}
\chi(X) & \text{if $n$ is even},\\
0 & \text{if $n$ is odd}.
\end{cases}
\]
\end{theorem}
\begin{proof}
Immediate from the proof of Theorem 2.
\end{proof}

\begin{theorem}
Suppose $X$ is an $n$-dimensional compact smooth manifold
and $V$ is a continuous tangent vector field on $X$.
If $V$ and $\partial^\perp V$ have only isolated zeros, 
then the following equality holds:
\[
\ind^*_\tau(V)=
\begin{cases}
\chi(X) & \text{if $n$ is even},\\
\chi(X) - \chi(\partial_- X) & \text{if $n\ge 3$},\\
\chi(X)-\frac{1}{2}\chi(\partial_0 X) - \chi(\partial_- X) & \text{if $n=1$}.
\end{cases}
\]
\end{theorem}

\begin{proof}
The only difference between Theorem 2 and Theorem 4 is the existence of the
isolated zeros of $\partial^\perp V$ that are not the zeros of $V$. 
Since there is nothing to prove when $n=1$, we assume that $n>1$.

Suppose $n$ is even. 
There are three types of zeros of $\tilde D V$, not two; 
the third type is an isolated zero $(0,x)$ corresponding to $p=(1,x)$
such that $V(p)$ is a non-zero tangent vector of $\partial X$ as mentioned above.
The local index of $\tilde D V$ is $2\ind^*_\tau(V,p)$.
Therefore the Poincar\'e-Hopf formula for $\tilde D V$ gives
$2\ind^*_\tau(V)=2\chi(X)$.

Next suppose $n\ge 3$.
Follow the proof of Theorem 2, treating the zeros of $\partial^\perp V$
like the zeros of $V$ on the boundary, and apply the Poincar\'e-Hopf theorem.
\end{proof}

\bibliographystyle{amsplain}

\end{document}